\documentclass{amsart}

\usepackage[usenames]{color}
\usepackage{soul}
\usepackage{a4,amsmath,amssymb,amsthm,amsfonts,enumerate,url,mathbbol,mathrsfs,tikz,graphicx,pifont,}
\usepackage{hyperref}
\usepackage{pgfplots}
\usepgfplotslibrary{polar}

\newcommand*\diff{\mathop{}\!\mathrm{d}}

\usetikzlibrary{automata,positioning}
	
\usepackage[utf8]{inputenc}

\usepackage{verbatim}
\usepackage{array}
\usepackage[normalem]{ulem}

\DeclareMathOperator{\diag}{diag}

\newtheorem{lemma}{Lemma}[section]
\newtheorem{cor}[lemma]{Corollary}
\newtheorem{theo}[lemma]{Theorem}
\newtheorem{rem}[lemma]{Remark}
\newtheorem{prop}[lemma]{Proposition}

\def\:{\thinspace:\thinspace}

\newcommand{\R}{\mathbb{R}}
\newcommand{\C}{\mathbb{C}}

\newcommand{\Z}{\mathbb{Z}}

\DeclareMathOperator{\esssup}{ess\;sup}

\DeclareMathOperator{\sgn}{sgn}
\DeclareMathOperator{\Ker}{ker}

\numberwithin{equation}{section}

\begin{document}
\title[Analytic Solutions for Stochastic Hybrid Models]{Analytic Solutions for Stochastic Hybrid Models of Gene Regulatory Networks }


\author{Pavel Kurasov}
\address{Pavel Kurasov, Dept. of Mathematics, Stockholm Univ., 106 91 Stockholm, SWEDEN}
\email{kurasov@math.su.se}

\author{Delio Mugnolo}
\address{Delio Mugnolo, Chair of Analysis, University of Hagen, 58084 Hagen, GERMANY}
\email[corresponding author]{delio.mugnolo@fernuni-hagen.de}

\author{Verena Wolf}
\address{Verena Wolf, Computer Science Department, Saarland University, 66123 Saarbrücken, GERMANY}
\email{verena.wolf@uni-saarland.de}

\begin{abstract}
Discrete-state stochastic models are a popular approach to describe the inherent stochasticity of gene expression in single cells. The analysis of such models is hindered by the fact that the underlying discrete state space is extremely large.  Therefore hybrid models, in which protein counts are replaced by average protein concentrations, have become  a popular alternative.

The evolution of the corresponding probability density functions is given by a coupled system of hyperbolic PDEs. This system has Markovian nature but its hyperbolic structure makes it difficult to apply standard functional analytical methods. We are able to prove convergence towards the stationary solution and determine such equilibrium explicitly by combining abstract methods from the theory of positive operators and elementary ideas from potential analysis.
\end{abstract}

\maketitle

\section{Introduction}
Very small copy numbers of genes, RNA, or protein molecules in single cells give rise to heterogeneity across genetically identical cells~\cite{raj2008nature}. 
During the last decades discrete-state stochastic models  
have become a popular approach to  describe gene expression in single cells since they
 adequately account for discrete random events underlying such cellular processes (see \cite{Schnoerr17} for a review).
 Exact solutions for such models are available if they obey 
 detailed balance \cite{Laurenzi2000} or restrict to monomolecular intracellular interactions \cite{jahnke}. 
 However, since gene regulatory networks typically contain feedback loops,
  second-order  interactions are necessary for an adequate description.
 Moreover, neither   detailed balance   nor linear dynamics are realistic assumptions
   even for simple regulatory networks.
 Recently,   analytical solutions for single-gene feedback loops have been presented~\cite{GriSchNew12,hornos2005self,kumar2014exact,liu2016decomposition,vandecan2013self,visco2008exact}.

The underlying   state space of discrete-stochastic models 
that describe gene regulatory networks  
is typically extremely large due to the combinatorial nature of
molecule counts for different types of chemical species. Moreover,
realistic upper bounds on protein counts are often not known. 
Therefore, hybrid models have become popular in which for highly-abundant species only 
average  counts are tracked while discrete random variables
are  used to represent species with low copy numbers.
These hybrid approaches allow for faster and yet accurate 
Monte-Carlo sampling that
stochastically selects counts of species with low copy numbers
and numerically integrates average counts of all other species
\cite{hepp2015adaptive,crudu2009hybrid,Herajy2012,puchalka2004bridging,lin2018stochastic,bokes2013transcriptional}. 
{An application of this mathematical formalism to models in mathematical ecology can be found in~\cite{Costa2016}.}

Hybrid or fluid approaches have also been investigated in the 
context of stochastic Petri nets as a kind of mean field approximation, thus giving rise to fluid stochastic Petri nets~\cite{TriKul93,HorKulNic98}. {The underlying stochastic process of this formalism is a piecewise deterministic Markov processes (PDMP), as introduced by Davis in~\cite{Davis1984}. However, it seems that this connection was rarely investigated in the literature.}
Lück and the present authors have shown in~\cite{KurLucMug18} how this formalism can be successfully adapted to study gene regulatory networks. More precisely, a stochastic hybrid approach for gene regulatory networks has been proposed therein, in which the state of the genes is represented by a discrete-stochastic variable, while the evolution of the protein numbers is modeled by an ordinary differential equation. The main aim of~\cite{KurLucMug18} was to show  that {this class of 
PDMPs  allows} for more efficient numerical simulations than common, purely discrete master equations, while providing solutions that are as accurate as those provided in reference studies, like~\cite{GriSchNew12}. The scope of the present paper is more analytic: in particular, we discuss in details properties of the evolution equation  -- in fact, a coupled system of hyperbolic PDEs -- that {describes the 
protein production of an autocatalytic gene. 
Different models for this common gene regulatory motif  have been proposed
with some focusing on protein production that happens in bursts \cite{bokes2018high,chen2020limit}, leading to systems with discontinuous trajectories.
Other models describe switching between low and high production by   Hill functions \cite{lin2016gene,lin2016bursting} or rely on a Poisson representation \cite{herbach2019stochastic,lin2019exact}.
 }

{The model that we consider here is a PDMP with a single continuous
	variable for the protein concentration and two modes of production. 
The corresponding conditional densities are } mathematically described by the system of first-order partial differential equations
\begin{equation}\label{eq:system}
\frac{\partial}{\partial t}
\begin{pmatrix} f_1(x,t)\\ f_2(x,t)\end{pmatrix} = -
\frac{\partial}{\partial x} \left[
\begin{pmatrix}
a-bx & 0 \\
0 & c-dx \end{pmatrix}
\begin{pmatrix} f_1(x,t)\\ f_2(x,t)\end{pmatrix} \right]+ \begin{pmatrix}
- \lambda (x) & \mu (x) \\
\lambda (x) & - \mu (x) \end{pmatrix}
\begin{pmatrix} f_1(x,t)\\ f_2(x,t)\end{pmatrix} .
\end{equation}
The basic setting is outlined in Section~\ref{sec:formulation}.

{Due to the terms $a-bx,c-dx$ in~\eqref{eq:system}, t}his system has to be complemented by suitable boundary and initial conditions: {this subtlety is often omitted in papers dealing with similar models. Correct boundary conditions for models rather similar to ours (and some generalizations thereof) have been already derived in~\cite{faggionato2009non,zeiser2010autocatalytic}.}
We will show that {already} the requirement of the solution to be a probability distribution -- and in particular a positive-valued $L^1$-function -- determines a specific choice of boundary conditions. Furthermore, we are interested in calculating analytically the equilibrium  distribution. These distributions can be computed numerically in a few special cases but in order 
to understand their properties it is useful to have explicit formulas for the solutions: this is done in Section~\ref{sec:balance-1d}. Results that are comparable to ours have been obtained in the literature: we especially mention~\cite{faggionato2009non,zeiser2010autocatalytic}.
{{On one hand, more general models have been studied in~\cite{faggionato2009non}, but the
explicit character of our model allows us to clarify the role of boundary conditions which are (in our case) derivable from the equation itself and remove all normalization issues: we believe that our studies show a way how
most general problems formulated in~\cite{faggionato2009non} can be attacked. On the other hand, \cite{zeiser2010autocatalytic} opts for a more detailed analysis of four models, which correspond to special cases of~\eqref{eq:system} involving specific choices of $\lambda,\mu$ that are only allowed to be either constant (``unregulated promoter type''), or else linear or affine functions; as well as of $a,b,c,d$.
Indeed,  both switching rates $\lambda=\lambda(x)$ and $\mu=\mu(x)$ of our model depend on the protein concentration $x$ (in the jargon of~\cite{zeiser2010autocatalytic}, we are allowing for autocatalytic networks), while other models, for which analytical solutions of the equilibrium  distribution have been derived, allow for constant switching rates~\cite{friedman2006linking,bena2006dichotomous,hufton2016intrinsic} only.}}

The analytic features of the time-dependent differential equations in~\eqref{eq:system} are subtler than one may believe at a first glance: the innocent looking coupling turns a system of hyperbolic equations (which can be explicitly solved and whose solutions would otherwise become extinct in finite time) into a time-continuous Markov chain, which we are going to study by the classical theory of strongly-continuous semigroups of operators on an appropriate $L^1$-space. Proving well-posedness of the associated Cauchy problem is straightforward; {indeed, the stochastic process generated driving~\eqref{eq:system} can be constructed explicitly in this and some more general settings briefly discussed in Remarks~\ref{rem:linfty} and~\ref{rem:irred}, see~\cite{BenLebMal15}. However it seems that determining the correct domain of the generator of this stochastic process was an open problem to date (Faggionato et al.\ write e.g.: ``As discussed in~\cite{Davis1993}, if the jump rates [...] are not uniformly bounded, it is a difficult task to characterize exactly the domain [...] of the generator [...]'' {\cite[p.~265]{faggionato2009non}}): we solve it in Proposition~\ref{lem:semigracal}, see also Lemma~\ref{lem:dissip}.

Remarkably, d}etermining the long-time behavior of this $C_0$-semigroup turns out to be more elusive than one could naively conjecture: {this was already remarked in~\cite[\S~4.1]{zeiser2010autocatalytic}, where numerical simulations are run for rather specific choices of parameter}. One expects the system to converge to its steady-state solution, and this guess is indeed correct. More precisely, it is a natural goal to prove that the differential operator matrix acting on a vector-valued $L^1$-space that appears in~\eqref{eq:system} has no eigenvalue on the imaginary axis other than 0, thus excluding oscillatory behavior. However, the most classical techniques we have tried to apply to achieve this task fall short off the mark: in Section~\ref{sec:time-dep} we develop a method that seems to be new in the literature and of independent interest. It is based on a combination of classical Perron--Frobenius theory for positive semigroups and some recent compactness results on kernel operators on $L^1$-spaces. {An estimate of the rate of convergence to the steady-state solution was obtained in~\cite{BenLebMal12}, but  with respect to the much weaker topology induced by a Wasserstein distance}.

In Section~\ref{SecLast} we finally compare our abstract results with numerical simulations, finding that they are in good agreement.

%


\section{Formulation of the model}\label{sec:formulation}
On the domain 
\begin{equation}
x \in \left[ \frac{a}{b}, \frac{c}{d} \right], \; \; t \geq 0, 
\end{equation}
 consider the evolution problem determined by
equation~\eqref{eq:system}. To ensure that the interval $ [ \frac{a}{b}, \frac{c}{d} ] $ is non-degenerate, we assume that 
\begin{equation} \label{ass1}
ad < bc
\end{equation}
 and  that $ b \neq 0, d \neq 0$. (In equations deriving from biological models it is natural to interpret $a,b,c,d$ as strictly positive rates. However, this is not relevant for the purpose of our analysis, and in fact, in Section~\ref{SecLast} we are going to study some toy models involving negative rates, too.)

The functions $ \lambda, \mu$ are assumed to be continuous and positive, i.e.,
\begin{equation} \label{ass2}
\begin{array}{l}
\displaystyle
\lambda, \mu \in C \left(\left[\frac{a}{b}, \frac{c}{d}\right]\right), \\[3mm]
\displaystyle \lambda (x) >  0\hbox{and } \mu (x) > 0 \hbox{for all }x\in \left[\frac{a}{b}, \frac{c}{d}\right]\ .
\end{array}
\end{equation}
Since $\lambda,\mu$ are positive on a compact interval, it follows that their minimum is larger than $0$, too.
This assumption on $\lambda,\mu$ may be relaxed, see Remark~\ref{rem:linfty}, but we impose it in order to simplify our presentation. 

We assume that initial conditions at
 $ t= 0 $ are provided
\begin{equation}
f_j (x,0) = f_{0j} (x), \; \, j =1,2. 
\end{equation}

We are going to show that the evolution equation is governed by a positivity preserving $C_0$-semigroup, so that positive initial data determine positive solutions $(f_1(\cdot,t),f_2(\cdot,t)) $ for any $ t > 0. $
Moreover, it is straightforward to prove that if the positive vector-valued function $(f_1,f_2)$ solves~\eqref{eq:system}, then
$$
\int_{a/b}^{c/d} \left( f_1 (x,t) + f_2(x,t) \right) \diff x,\qquad t\ge 0,
$$
is an integral of motion, i.e., it is preserved over time under the evolution of~\eqref{eq:system}:	 this is equivalent to saying that the semigroup is stochastic and can be proven by integration by parts, see
Proposition \ref{lem:semigracal} below. Hence the function $(f_1,f_2)$ can be interpreted as (vector-valued) density assuming the normalization condition
\begin{equation} \label{normcond}
\int_{a/b}^{c/d} \left( f_1 (x,t) + f_2(x,t) \right) \diff x = 1\qquad \hbox{for all }t\ge 0\ .
\end{equation} 
Therefore it is natural to look for solutions in the space $ L^1 (\frac{a}{b}, \frac{c}{d}). $

Then the maximal domain for the semigroup generator $\mathcal A$ given by
\begin{equation}
\mathcal Af(x) :=   -
\frac{\partial}{\partial x} \left[
\begin{pmatrix}
a-bx & 0 \\
0 & c-dx \end{pmatrix}
\begin{pmatrix} f_1(x,t)\\ f_2(x,t)\end{pmatrix} \right] + 
\begin{pmatrix}
- \lambda (x) & \mu (x) \\
\lambda (x) & - \mu (x) 
\end{pmatrix}
\begin{pmatrix} f_1(x,t)\\ f_2(x,t)\end{pmatrix}
\end{equation}
consists of functions from  $ L^1 (\frac{a}{b}, \frac{c}{d}) $ such that {both
$$x\mapsto \frac{\partial}{\partial x} (a-b x) f_1(x)  ,\qquad x\mapsto  \frac{\partial}{\partial x} (c-dx) f_2(x)
$$
belong to $L^1 \left(\frac{a}{b}, \frac{c}{d}\right)$, too. }
In particular this implies that if $(f_1,f_2)$ is a solution to~\eqref{eq:system} for initial data
\[
f_{01},f_{02}\in W^{1,1} \left(\frac{a}{b}, \frac{c}{d}\right)\ ,
\]
then both  $ f_1 ,f_2$ belong to the Sobolev space $ W^{1,1} $
 on any compact sub-interval, more precisely
\begin{equation}
f_1 \in W^{1,1} \left( \frac{a}{b} + \epsilon, \frac{c}{d}\right), \; \;  f_2 \in W^{1,1} \left( \frac{a}{b}, \frac{c}{d} - \epsilon\right), \; \, \forall \epsilon > 0.
\end{equation}
Observe that for all $ x\in (\frac{a}{b}, \frac{c}{d})$ the coefficient matrix 
\begin{equation}\label{eq:matrix}
\begin{pmatrix}
bx-a & 0 \\
0 & dx-c \end{pmatrix}
\end{equation}
is indefinite, since $bx-a>0$ and $dx-c<0$ and non-singular inside $(\frac{a}{b},\frac{c}{d})$. 
Because for all $x$ in the relevant interval $bx-a>0$ and $dx-c<0$ one may guess that the correct boundary conditions are to be imposed on the right endpoint on the interval for $f_1$, and on the left endpoint on the interval for $f_2$.
In the following, we are going to show that this is, in fact, necessarily the case.
We are going to obtain these boundary conditions
by investigating the possible stationary state, see Theorem \ref{Theo}.

\section{Search for the stationary state}\label{sec:balance-1d}

Numerical simulations of the system~\eqref{eq:system} show that it always tends to equilibrium, therefore it looks natural to start our analysis from investigating a possible stationary solution and calculating it explicitly. A stationary solution should satisfy the
system of ordinary differential equations:
\begin{equation}\label{eq:stationary}
-
\frac{d}{d x} \left[
\begin{pmatrix}
a-bx & 0 \\
0 & c-dx \end{pmatrix}
\begin{pmatrix} \psi_1(x)\\ \psi_2(x)\end{pmatrix} \right] + 
\begin{pmatrix}
- \lambda (x) & \mu (x) \\
\lambda (x) & - \mu (x)
\end{pmatrix}
\begin{pmatrix} \psi_1(x)\\ \psi_2(x)\end{pmatrix} = 0.
\end{equation}
Under our assumptions any stationary solution has to be a {weakly} differentiable function
inside the interval $ (\frac{a}{b}, \frac{c}{d} ) $ but may have singularities at
the boundary points. Nonetheless the singularities cannot be too strong, since the densities have to be integrable functions. Let us prove
the following elementary statement to be used in  what follows to deduce certain necessary boundary conditions.

\begin{lemma}\label{lem:intbdd}
Let $ \psi $ be a positive continuous integrable function defined on the interval $ (0, 1) $. Then there exists
a sequence  $(x_n)_{n\in \mathbb N}\subset (0,1)$ such that
\begin{equation}
\lim_{n\to\infty}x_n= 0\quad\hbox{and}\quad \lim_{n \rightarrow \infty } x_n \psi (x_n) = 0.
\end{equation}
\end{lemma}
\begin{proof}
Let us present such a sequence explicitly. Consider the intervals $ [ \frac{1}{2^n}, 2 \frac{1}{2^n} ]$
and denote by $ x_n \in [ \frac{1}{2^n}, 2 \frac{1}{2^n} ] $ one of the minimum points for $ \psi $ in the interval: so  
$$\psi (x) \geq \psi (x_n)\quad \hbox{whenever } x \in [ \frac{1}{2^n}, 2 \frac{1}{2^n} ] $$
If $ x_n \psi (x_n) $ tends to zero, then we have such a sequence. Let us now assume that there is no subsequence of $( x_n  \psi (x_n) )_{n\in \mathbb N} $ tending to zero:
then there is a positive number $ \delta > 0 $ such that
$$ x_n \psi (x_n ) > \delta$$
for all sufficiently large $ n$.
But then it follows that $ \psi $ satisfies the lower estimate
$$\psi (x) \geq \delta 2^n\quad \hbox{whenever } x \in [ \frac{1}{2^n}, 2 \frac{1}{2^n} ] $$
and hence cannot be integrable since
$$ \int_0^1 f(x) \diff x >  \sum_{n=1}^\infty \delta \, 2^n \frac{1}{2^{n+1}} = \infty. $$
This contradiction proves the lemma. 
\end{proof}

Lemma~\ref{lem:intbdd} applied to $ \psi_1 $ and $ \psi_2 $ implies that there exist sequences $ x_n^- \rightarrow  \frac{a}{b},  x_n^+ \rightarrow \frac{c}{d} $
such that the following  two limits hold
\begin{equation}\label{eq:balance}
 \left\{
\begin{array}{l}
\displaystyle \lim_{n \rightarrow \infty} (a-bx_n^-) \psi_1 (x_n^-) = 0\ , \\[3mm]
\displaystyle \lim_{n \rightarrow \infty} (c-d x_n^+) \psi_2 (x_n^+) = 0\ .
\end{array} 
\right. 
\end{equation}
Let us now return back to the differential system~\eqref{eq:stationary} and sum up the two equations to get
\begin{equation}
\frac{d}{d x} \left[ (a-bx) \psi_1 (x) + (c-dx ) \psi_2 (x) \right] = 0,
\end{equation}
which implies that 
\begin{equation} \label{eqK}
(a-bx) \psi_1 (x) + (c-dx ) \psi_2 (x) \equiv K
\end{equation} 
for some $K\in \mathbb R$.
Let us determine this constant by evaluating the function near $ x = \frac{a}{b} $ and $ x = \frac{c}{d} $ using
the sequences $ x_n^\pm $ introduced above:
\[
\begin{array}{ccl}
K & = & \displaystyle  \lim_{n \rightarrow \infty} \left( (a-bx_n^-) \psi_1 (x_n^-) + (c-dx_n^- ) \psi_2 (x_n^-) \right)  \\
& = &
\displaystyle 0 +   \lim_{n \rightarrow \infty} \left(  (c-dx_n^- ) \psi_2 (x_n^-) \right)  \\
 & \geq & 0; \\[3mm]
K & = &\displaystyle   \lim_{n \rightarrow \infty} \left( (a-bx_n^+) \psi_1 (x_n^+) + (c-dx_n^+ ) \psi_2 (x_n^+) \right)  \\
& = & 
\displaystyle  \lim_{n \rightarrow \infty} \left( (a-bx_n^+) \psi_1 (x_n^+)  \right)  + 0 \\
& \leq &   0.
\end{array} 
\]
It follows that $ K = 0 $ and therefore the functions $ \psi_{1,2} $ (continuous inside the open interval)  tend to zero limits at the right and left boundary 
points respectively
\begin{equation}
\psi_1 \left(\frac{c}{d}\right) = \psi_2 \left(\frac{a}{b}\right) = 0 .
\end{equation}
Let summarize our studies.

\begin{theo} \label{Theo}
Every positive integrable solution $ (\psi_1, \psi_2) $ to the system~\eqref{eq:stationary} is continuous on the
open interval $ (\frac{a}{b}, \frac{c}{d}) $ and satisfies the following conditions at the boundary points:
\begin{equation} \label{eqbc}
\begin{array}{ll}
\lim_{x \rightarrow a/b} (a-bx) \psi_1 (x) = 0 , & \lim_{x \rightarrow c/d} \psi_1 (x) = 0, \\[3mm]
\lim_{x \rightarrow a/b}  \psi_2 (x) = 0 , & \lim_{x \rightarrow c/d} (c-dx)  \psi_2 (x) = 0. \\
\end{array}
\end{equation}
\end{theo}
\begin{proof}
To accomplish the proof we just need to show that the limits $ \lim_{x \rightarrow a/b} (a-bx) \psi_1 (x) = 0 $ and
$  \lim_{x \rightarrow c/d} (c-dx)  \psi_2 (x) = 0 $ hold, not just along subsequences as in Lemma \ref{lem:intbdd}. This follows
directly from~\eqref{eqK} and the fact that $ K = 0 $
$$ (a-bx) \psi_1 (x) + (c-dx) \psi_2 (x) = 0 .$$
Taking into account that $ \psi_1 (x) $ tends to zero at $ x = c/d $, {and that so does $ \psi_2 (x) $ at $ x= a/b $}, we arrive at~\eqref{eqbc}.
\end{proof}

We observe that it is not convenient to keep working with two density functions, due to an explicit relation between
them. Let us namely consider a solution $ (\psi_1, \psi_2) $ to the system~\eqref{eq:stationary} and introduce a new positive continuously differentiable function
\begin{equation}
h(x) := (bx-a) \psi_1 (x) =  (c-dx) \psi_2(x).
\end{equation}
Hence the function $ h $ is continuous in the closed interval $ [ \frac{a}{b}, \frac{c}{d} ] $ and
satisfies Dirichlet conditions at both endpoints:
\begin{equation} \label{eqbch}
h\left( \frac{a}{b}\right) = 0 = h\left(\frac{c}{d}\right) . 
\end{equation}

Taking the difference between the two equations in~\eqref{eq:stationary} we get the following single differential
equation on the function $ h $ introduced above:
\begin{equation}
\frac{d}{dx} h (x) =  \left( \frac{\lambda(x)}{bx-a} - \frac{\mu(x)}{c-dx} \right) h (x).
\end{equation}
We can solve this equation analytically whenever we can integrate the function in brackets on the right hand side
\begin{equation} \label{eqhh}
h(x) = K \exp \left\{\int_{x_0}^x \left(\frac{\lambda(y)}{by-a} - \frac{\mu(y)}{c-dy} \right) \diff y \right\}\ ,
\end{equation}
{for arbitrary} $ x_0 \in  (\frac{a}{b},\frac{c}{d}) $  and $ K \in (0, \infty)$.
The functions $ \mu $ and $ \lambda $ are positive
 definite  on a compact interval and therefore $ \mu (x) > C$, $\lambda (x) > C$, where $ C > 0 $ is a certain positive constant.
Hence the integral tends to $ - \infty $ at both endpoints of the interval. To see this, let us split
the integral as $ \int_{x_0}^x \frac{\lambda(y)}{by-a} \diff y - \int_{x_0}^x \frac{\mu(y)}{c-dy} \diff y. $ The second integral is bounded near $ x = a/b $, while
the integrand in the first integral can be estimated as $  \frac{\lambda(y)}{by-a} \geq \frac{C}{by -a}. $ Hence the 
following limit holds 
$$ \lim_{x \rightarrow a/b} \int_{x_0}^x \frac{\lambda (y)}{by -a} \diff y = - \lim_{x \rightarrow a/b} \int_{x}^{x_0} \frac{\lambda (y)}{by -a} \diff y < - \lim_{x \rightarrow a/b} \int_x^{x_0} \frac{C}{by -a} \diff y = - \infty  $$
$$ \Rightarrow  \lim_{x \rightarrow a/b} \int_{x_0}^x \frac{\lambda (y)}{by -a} \diff y = - \infty. $$
It follows that $ h $ given by~\eqref{eqhh} satisfies Dirichlet condition at $ x = a/b. $ Similarly, it satisfies Dirichlet condition at the opposite endpoint as well.

{
\begin{rem}\label{rem:hill}
(1) The assumption that $ \lambda(x)>0 $ and $\mu(x)>0$ for all $x\in \left[ \frac{a}{b}, \frac{c}{d} \right] $ is crucial to ensure that the integral {in~\eqref{eqhh}} tends to $ - \infty $
at \textit{both} endpoints, and therefore that the function $ h $ vanishes there for any choice of $ K$. If the integral would not tend to $ - \infty $, then
the exponential function would not go to zero and to satisfy \eqref{eqbch} it is unavoidable to let $ K= 0 $: this means that the only stationary state is
the constant zero function. {If for example the switching rates $\lambda,\mu$ are taken to be
\begin{equation}\label{eq:hill}
\lambda(x)=\frac{\beta_1 (x-\frac{a}{b})^n}{k_1^n+(x-\frac{a}{b})^n},\quad \mu(x)=\frac{\beta_2 (\frac{c}{d}-x)^m}{k_2^m+(\frac{c}{d}-x)^m},\qquad x\in \left[\frac{a}{b},\frac{c}{d}\right],
\end{equation}
the integral fails to diverge near $\frac{a}{b}$ (resp., near $\frac{c}{d}$) if and only if $n > 0$ (resp., $m > 0$). If in particular $m,n = 1,2, \dots $ and hence $\lambda,\mu$ in~\eqref{eq:hill} are  Hill-type functions, then there is no non-trivial stationary state.}

(2) On the other hand, positive definiteness of $ \lambda $ and $ \mu $ can be relaxed by merely requiring that
\begin{itemize}
\item $ \lambda (x)\ge 0$ and $ \mu(x) \ge 0$ for all $x\in \left[ \frac{a}{b}, \frac{c}{d} \right] $ and 
\item the integrals
$$ \int_{\frac{a}{b}}^{\frac{a}{b} + \epsilon} \frac{\lambda(y)}{b y -a} dy  \quad  \mbox{and} \quad \int_{\frac{c}{d} - \epsilon}^{\frac{c}{d}}
\frac{\mu(y)}{c - dy} dy $$
are diverging for a certain $ \epsilon > 0$. 
\end{itemize}
{(These conditions are e.g.\ satisfied if one allows} $ \lambda $ to behave like
$$ \lambda (x) \sim - \frac{1}{\ln (x- \frac{a}{b})} \quad \hbox{as }x \rightarrow \frac{a}{b}, $$
likewise for $\mu(x)$ as $x\to \frac{c}{d}$.)
{Under these assumptions, the above reasoning still apply and show that non-trivial stationary states exist.}
\end{rem}
}

Summing up, we have obtained that the densities of the stationary solution of~\eqref{eq:system} can be explicitly computed as follows.

\begin{theo}\label{thm:pavel-densities}
Assume that functions $ \lambda,\mu $ on the interval $ [\frac{a}{b},\frac{c}{d}]$ satisfying~\eqref{ass2} are given.  Then up to the parameter $K$ there is a unique {integrable,} strongly positive\footnote{A function is called \textit{strongly positive} if it is positive outside a set of measure zero \cite{ReeSim78}.}  solution to the stationary equation ~\eqref{eq:stationary}: it is given by 
\begin{equation} \label{solution1}
\begin{split}
\psi_1(x) & = \displaystyle \frac{K}{bx -a}   \exp \left\{\int^x \left(\frac{\lambda(y)}{by-a} - \frac{\mu(y)}{c-dy} \right) \diff y \right\}\ , \\
\psi_2(x) & = \displaystyle \frac{K}{c - dx}  \exp \left\{\int^x \left(\frac{\lambda(y)}{by-a} - \frac{\mu(y)}{c-dy} \right) \diff y \right\}\ ,
\end{split}
\qquad x\in \left(\frac{a}{b},\frac{c}{d}\right).
\end{equation}
The solution always satisfies boundary conditions~\eqref{eqbc}.
\end{theo}

Note that despite $ h $ satisfies Dirichlet conditions at both endpoints, the densities $ \psi_{1,2} $ may have singularities there (see Section \ref{SecLast} for 
illuminating examples), {{as long as these are integrable. Formula~\eqref{solution1} agrees with \cite[(19) and (20)]{zeiser2010autocatalytic} in the special cases of affine $\lambda$ and constant $\mu$. On the other hand, \cite[(18)]{zeiser2010autocatalytic} is an example of a non-integrable stationary solution, which can therefore not be interpreted as a probability distribution: indeed, it features a switching rate $\lambda$ that  does not satisfy the relaxed conditions from Remark~\ref{rem:hill}.(2), as it is a linear function vanishing at $\frac{a}{b}$.}}

\section{Analysis of the time-dependent system}\label{sec:time-dep}

Let us now finally turn to the study of the system of time-dependent partial differential equations~\eqref{eq:system}. Each of these two equations models the time evolution  of a two-state continuous-time Markov chain, the vector-valued function $f$ thus representing a probability distribution.

This equation is meaningful for all $t\in \mathbb R$, but we will study in particular the evolution for $t\to \infty$ in the dependence from the configuration of the system for $t=0$. 

As in the rest of the article, throughout this section we are still imposing the conditions stated in Section~\ref{sec:formulation}.

\subsection{Preliminaries}
%
The partial differential equation~\eqref{eq:system} describes transport phenomena: Indeed, (probability) mass is shifted to the left or to the right depending on whether the coefficient of the first derivative is positive or negative, respectively.

On the other hand,~\eqref{eq:system} arises as a stochastic model and in the biologically relevant case of $\lambda>0$, $\mu>0$ the system is steadily driven by a mixing force described by the dynamical system
\begin{equation}\label{eq:system-3}
\frac{d}{d t}
\begin{pmatrix} f_1(x,t)\\ f_2(x,t)\end{pmatrix} =\mathcal A_2 \begin{pmatrix} f_1\\ f_2\end{pmatrix}(x,t):= 
\begin{pmatrix}
- \lambda (x) & \mu (x) \\
\lambda (x) & - \mu (x)
\end{pmatrix}
\begin{pmatrix} f_1(x,t)\\ f_2(x,t)\end{pmatrix}
\end{equation}
that competes with the transport (with space-dependent speed) given by the vector-valued partial differential equation
\begin{equation}\label{eq:system-4}
\begin{split}
\frac{\partial}{\partial t}
\begin{pmatrix} f_1(x,t)\\ f_2(x,t)\end{pmatrix} =\mathcal A_1 \begin{pmatrix} f_1\\ f_2\end{pmatrix}(x,t):=&
\begin{pmatrix}
bx-a & 0 \\
0 & dx-c \end{pmatrix}
\frac{\partial}{\partial x} \begin{pmatrix} f_1(x,t)\\ f_2(x,t)\end{pmatrix}  \\
&+\begin{pmatrix}
b & 0 \\
0 & d \end{pmatrix}
 \begin{pmatrix} f_1(x,t)\\ f_2(x,t)\end{pmatrix} 
\ .
\end{split}
\end{equation}

\begin{rem}
In general the two operator matrices $\mathcal A_1,\mathcal A_2$ do not commute, but by~\cite[Exer.~III.5.11]{EngNag00} the solution $f$ to the complete time-dependent problem \eqref{eq:system} can be given by means of the Lie--Trotter product formula
\begin{equation}\label{eq:lie-trotter}
f(t)=e^{t\mathcal A}f_0 \stackrel{!}{=}\lim_{n\to \infty} \left(e^{\frac{t}{n}\mathcal A_1}e^{\frac{t}{n}\mathcal A_2} \right)^n f_0\ ,
\end{equation}
where $\left(e^{t \mathcal A_1}\right)_{t\ge 0}$ and $\left(e^{t\mathcal A_2}\right)_{t\ge 0}$ are the $C_0$-semigroups generated by the operators $\mathcal A_1,\mathcal A_2$
and of course $f_0$ is the initial data.
The operator splitting in~\eqref{eq:lie-trotter} seems to be numerically interesting, since the linear dynamical system~\eqref{eq:system-3} is solved by
\begin{equation}\label{eq:matrixexpo}
e^{t\mathcal A_2}:f_0\mapsto \begin{pmatrix}
{\frac {\lambda\,{{\rm e}^{-t \left( \mu+\lambda \right) }}+\mu}{\mu+\lambda}}&
{\frac {\mu\, \left(1- {{\rm e}^{-t \left( \mu+\lambda \right) }} \right) }{\mu+\lambda}}\\ \noalign{\medskip}
{\frac {\lambda\, \left(1- {{\rm e}^{-t \left( \mu+\lambda \right) }} \right) }{\mu+\lambda}}&
{\frac {\mu\,{{\rm e}^{-t \left( \mu+\lambda \right) }}+\lambda}{\mu+\lambda}}
\end{pmatrix}\begin{pmatrix}f_{01}\\ f_{02}\end{pmatrix},\qquad t\ge 0\ ,
\end{equation}
where $f_{01},f_{02}:(\frac{a}{b},\frac{c}{d})\to \mathbb R$ are the two components of the initial data $f_0$. 
Also the solutions to the pure transport equation~\eqref{eq:system-4} can be determined, at least in principle: in the case $b>0,d>0$ that is relevant for our model \emph{one} solution of~\eqref{eq:system-4} is given by
\begin{equation}\label{eq:explicit}
\begin{split}
\begin{pmatrix}f_1\\ f_2\end{pmatrix}: (x,t)\mapsto 
\begin{pmatrix}
e^{tb} f_{01}\left( \left(x-\frac{a}{b}\right)e^{tb}+\frac{a}{b}\right)\\
\noalign{\medskip}
e^{td} f_{02}\left(\left( x-\frac{c}{d}\right)e^{td}+\frac{c}{d}\right)
\end{pmatrix}\ .
\end{split}
\end{equation}
 While this formula does not bother to take into account possible boundary conditions, and indeed it is not defined for all $(x,t)$, it gives a hint about the qualitative behavior of the solution: namely, the profile of the initial data is rescaled and concentrated as the time passes by and the spacial argument is deformed according to the laws
 \[
x\mapsto   \left(x-\frac{a}{b}\right)e^{tb}+\frac{a}{b},\quad  x\mapsto\left( x-\frac{c}{d}\right)e^{td}+\frac{c}{d},\qquad t\ge 0\ .
 \]
(Observe that the intervals $(\frac{a}{b},+\infty)$ and $(-\infty,\frac{c}{d})$ are left invariant under these monotone transformations for all $t\ge 0$. We also note that the speed of propagation approaches 0  as $x\to \frac{a}{b}$ or $x\to \frac{c}{d}$, respectively.)
 
Accordingly, variance diminishes as the total probability is conserved but its profile gets shifted to the left in the first and to the right in the second equation of~\eqref{eq:system-4}, respectively; this is in sharp contrast with the case of $b=0=d$ for the characteristics of the hyperbolic systems are straight lines and hence mass is steadily leaving the system. 
\end{rem}



As already mentioned, the system of partial differential equations in~\eqref{eq:system} describes the limiting case of a Markov chain. For this reason, one expects convergence to an equilibrium given by the solution to the stationary equation~\eqref{eq:stationary}
which we have investigated in Section~\ref{sec:balance-1d}. But to begin with, let us first analyze the time-dependent problem.
Our analysis will be based on properties of the Banach spaces $L^1(\frac{a}{b},\frac{c}{d};\mathbb R^2)$ and  $L^\infty(\frac{a}{b},\frac{c}{d};\mathbb R^2)$, normed as
\[
\|f\|_1:=\int_\frac{a}{b}^\frac{c}{d}\left(|f_1(x)|+|f_2(x)|\right)\diff x \quad \hbox{and}\quad \|f\|_\infty:=\esssup\limits_{\frac{a}{b}\le x\le \frac{c}{d}} \max\{|f_1(x)|,|f_2(x)|\}\ ;
\]
they are, in fact, Banach lattices whenever endowed with the natural pointwise ordering induced by the ordering in $\mathbb R^2$. This will be important, as our analysis on the long-time asymptotics of~\eqref{eq:system} will be based on properties of positivity preserving operator semigroups.

It is natural to perform the analysis of the evolution equation~\eqref{eq:system} in an $L^1$-space, since we are interested in a stochastic model and the unknown $\left(f_1,f_2\right)$ represents a probability density.

Recall that a semigroup on an $L^1$-space is called \textit{irreducible} if its generator $A$ satisfies
\[
0\not\equiv g\ge 0 \hbox{and } \lambda f-Af=g\quad \Rightarrow \quad f>0 \hbox{a.e.}
\]
for some $\lambda>s(A)$.

\begin{lemma}\label{lem:dissip}
Let $\alpha\neq\beta$ be real numbers, $\alpha<\beta$. Let $p\in W^{1,1}(\alpha,\beta)$ such that
\begin{enumerate}[(i)]
\item either $p(s)>0$ for all $s\in (\alpha,\beta)$ and $p(\alpha)=0$
\item or $p(s)<0$ for all $s\in (\alpha,\beta)$ with $p(\beta)=0$
\end{enumerate}
and such that $\frac{1}{p}\in L^1(\alpha,\beta)$.
Then the operator
\[
A:f\mapsto \frac{d}{dx}\left(pf\right)
\]
with domain
\begin{itemize}
\item either $D(A):=\{f\in L^1(\alpha,\beta):(pf)'\in L^1(\alpha,\beta)\hbox{and }f(\beta)=0\}$ if $(i)$ holds,
\item or  $D(A):=\{f\in L^1(\alpha,\beta):(pf)'\in L^1(\alpha,\beta)\hbox{and }f(\alpha)=0\}$ if $(ii)$ holds,
\end{itemize}
generates an irreducible semigroup of positivity preserving contractions on $L^1(\alpha,\beta)$.

The embedding of $D(A)$ in $L^1(\alpha,\beta)$ is not compact, hence $A$ does not have compact resolvent.
\end{lemma}

Observe that $A$ is well-defined, since if $pf\in W^{1,1}(\alpha,\beta)$ and hence $pf$ is continuous on $[\alpha,\beta]$, so in particular its boundary values can be considered.

\begin{proof}
First of all, it is clear that the operator $A$ is closed and densely defined.
In view of~\cite[\S~C-II.1]{Nag86} and~\cite[Cor.~II.3.17]{EngNag00}, it suffices to show that $A$ is \textit{dispersive}, i.e.,
\[
\langle Af,\phi\rangle \equiv \int_\alpha^\beta Af(x)\phi(x)\ d\mu\leq 0\qquad \hbox{for all }f\in D(A)\hbox{and }\phi={\bf 1}_{\{f\ge 0\}}\ ,
\]
and that $A$ is $m$-dissipative, 
i.e., 
\begin{itemize}
\item for all $g\in L^1(\frac{a}{b},\frac{c}{d})$ there exists a solution $f\in D(A)$ of
\begin{equation}\label{eq:ode?}
(1+ p')f+ p\frac{df}{dx}=g
\end{equation}
\item and additionally
\[
\langle Af,\varphi\rangle \equiv \int_\alpha^\beta Af(x)\varphi(x)\ d\mu= 0\qquad \hbox{for all }f\in D(A)\hbox{and }\varphi=\sgn (f) \ . \\
\]
\end{itemize}
Indeed, the solution to~\eqref{eq:ode?} can be found by the variation of constants formula and is given by
\[
f(x)=e^{- \int_*^x \frac{1+ p'}{p}\ ds} \int_*^x e^{\int_*^s \frac{1+ p'}{p}\ dr}
\frac{g(s)}{p(s)}\ ds,\qquad x\in (\alpha,\beta)
\]
where $*=\beta$ or $*=\alpha$ depending on whether assumption $(i)$ or $(ii)$ is satisfied. This solution satisfies the prescribed boundary conditions and, in fact, $pf\in W^{1,1}(\alpha,\beta)$, so $f\in D(A)$. This explicit formula also shows that $f>0$ a.e.\ if $g\not\equiv 0$ is positive, hence $A$ generates an irreducible semigroup.

Furthermore,  there holds
\begin{equation}\label{eq:computeibp}
\int_\alpha^\beta \frac{d}{dx}\left(p(x)f(x)\right) \sgn(f(x))\ \diff x =\Big[ \left(p|f|\right)(x)\Big]_{x=\alpha}^{x=\beta }=0
\end{equation}
as well as
\[
\int_\alpha^\beta \frac{d}{dx}\left(p(x)f(x)\right) {\bf 1}_{\{f\ge 0\}}\ \diff x =\int_{\{f\ge 0\}}\frac{d}{dx}\left(p(x)f(x)\right) \ \diff x=\Big[ \left(pf\right)^+(x)\Big]_{x=\alpha}^{x=\beta }=0\ ,
\]
respectively, since  for a $ W^{1,1} $-function $ g $ one has
\begin{equation}\label{eq:gtsobolev}
(g^+)'=g'{\bf 1}_{\{g\ge 0\}}\ ,
\end{equation}
cf.~\cite[Lemma~7.6]{GilTru01}.

To conclude, let us show that the embedding of $D(A)$ in $L^1(\alpha,\beta)$ is not compact: we prove this assertion only for the case (i), the case (ii) being completely analogous. Pick a real sequence $(\beta_n)_{n\in \mathbb N}\subset [\alpha,\beta]$ with $\lim_{n\to\infty}\beta_n\to \beta$ and let
\[
f_n(x) := (1-\beta_n) x^{-\beta_n},\qquad n\in \mathbb N,\ x \in (\alpha,\beta).
\]
Then $f_n \in D(A)$ for all $n\in \mathbb N$ and indeed $\|f_n\|_1=1$ and also $\left(\|Af_n\|_1+\|f_n\|_1\right)_{n\in \mathbb N}$ is bounded. Therefore, if the embedding of $D(A)$ into $L^1(\alpha,\beta)$ was compact, then $(f_n)_{n\in \mathbb N}$ would have a convergent subsequence, say $(f_{n_k})_{k\in \mathbb N}$; let us denote its limit by $f_0$. Since $(f_n)_{n\in \mathbb N}$ has another subsequence that converges to $f_0$ almost everywhere. But $\lim_{n\to\infty}f_n(x)=0$ for a.e.\ $x\in (\alpha,\beta)$, so $f_0=0$ as well, a contradiction to the fact that $f_0$ is the limit of a sequence with unit $L^1$-norm.
\end{proof}

In fact, 
investigating the stationary state we have already seen that one needs to impose Dirichlet conditions 
\[f_1 \left(\frac{c}{d}\right) = 0 = f_2 \left(\frac{a}{b}\right). \] 
This is in accordance with the setting of Lemma~\ref{lem:dissip}. 

\begin{prop}\label{lem:semigracal}
Consider the operator $\mathcal A:=\mathcal A_1+\mathcal A_2$ defined by
\begin{equation}\label{eq:a1a2}
\mathcal A\begin{pmatrix}f_1\\ f_2\end{pmatrix}(x):= -\frac{d}{dx}
\begin{pmatrix}
a-b\cdot & 0\\ 0 & c-d\cdot
\end{pmatrix}
\begin{pmatrix}f_1\\ f_2\end{pmatrix}(x)+\begin{pmatrix}-\lambda(x) & \mu(x)\\ \lambda (x) & -\mu(x)\end{pmatrix}\begin{pmatrix}f_1\\ f_2\end{pmatrix}(x)
\end{equation}
with domain
\begin{equation}\label{eq:doma}
\begin{split}
D(\mathcal A)&:=\left\{f\equiv\begin{pmatrix}f_1\\ f_2\end{pmatrix}\in
L^1\left(\frac{a}{b},\frac{c}{d};\mathbb R^2 \right):\left(\begin{pmatrix}
a-b\cdot & 0\\ 0 & c-d\cdot
\end{pmatrix}
\begin{pmatrix}f_1\\ f_2\end{pmatrix}\right)'\in 
L^1\left(\frac{a}{b},\frac{c}{d};\mathbb R^2 \right)\right.\\
&\qquad \qquad\left.\hbox{and } f_1\left(\frac{c}{d}\right)=0,\ f_2\left(\frac{a}{b}\right) =0\right\}\ .
\end{split}
\end{equation}
Then $\mathcal A$ generates an irreducible $C_0$-semigroup of positivity preserving, stochastic contractions on $L^1\left(\frac{a}{b},\frac{c}{d};\mathbb R^2 \right)$. 
\end{prop}

\begin{proof}
Because $\lambda,\mu$ are $L^\infty$-functions, the multiplication operator $\mathcal A_2$ defined by the family of matrix functions
\[
M(x):=\begin{pmatrix}-\lambda(x) & \mu(x)\\ \lambda (x) & -\mu(x)\end{pmatrix},\qquad x\in \left(\frac{a}{b},\frac{c}{d}\right),
\] 
is bounded on  $L^1\left(\frac{a}{b},\frac{c}{d};\mathbb R^2 \right)$. In view of Lemma~\ref{lem:dissip} the unperturbed operator (corresponding to $\mathcal A_2\equiv 0$) generates an irreducible semigroup of positivity preserving operators on $L^1\left(\frac{a}{b},\frac{c}{d};\mathbb R^2 \right)$, hence also the full operator $\mathcal A$ generates a semigroup on $L^1\left(\frac{a}{b},\frac{c}{d};\mathbb R^2 \right)$. 

%

Positivity and irreducibility of the semigroup generated by $\mathcal A$ follow from the analogous property of the semigroup generated by $\mathcal A_2$, see~\eqref{eq:matrixexpo}, and by a product formula analogous to that in~\eqref{eq:lie-trotter}~\cite[Exer.~III.5.11]{EngNag00}.

Finally, take a positive initial data $f_0\in D(\mathcal A)$ and observe that
\begin{eqnarray*}
\frac{d}{dt}\|e^{t\mathcal A}f_0\|_{L^1}&=&\frac{d}{dt}\int_\frac{a}{b}^\frac{c}{d} (e^{t\mathcal A}f_0(x)|{\mathbb 1})_{\mathbb R^2}\ \diff x\\
&=&\int_\frac{a}{b}^\frac{c}{d} (\mathcal A e^{t\mathcal A}f_0(x)|{\mathbb 1})_{\mathbb R^2}\ \diff x\\
&=&\int_\frac{a}{b}^\frac{c}{d} (\mathcal A_1 e^{t\mathcal A}f_0(x)|{\mathbb 1})_{\mathbb R^2}\ \diff x+\int_\frac{a}{b}^\frac{c}{d} (\mathcal A_2 e^{t\mathcal A}f_0(x)|{\mathbb 1})_{\mathbb R^2}\ \diff x=0\ ,
\end{eqnarray*}
with respect to the inner product $(\cdot|\cdot)$ of $\mathbb R^2$, where in the last step the first integral vanishes integrating by parts as in~\eqref{eq:computeibp} and the second integral vanishes because $\mathbb 1$ lies in the null space of each matrix $M(x)^T$. By density, this shows that $\|e^{t\mathcal A}f_0\|_{L^1}=\|f_0\|_{L^1}$ for all $t\ge 0$ and all positive $L^1$-functions, i.e., $e^{t\mathcal A}$ is stochastic and hence also contractive.
\end{proof}

\begin{rem}\label{rem:linfty}
Our motivating model in~\eqref{eq:system} is based on \textit{one} gene with \textit{two} modes of expression. Analogous models also exist that describe ensembles of genes with three or more modes of expression, leading to $\R^n$-valued functions with $n\ge 3$~\cite{MieSzy13}. More precisely,~\eqref{eq:system} can be generalized to a system of differential equations driven by the operator
\begin{equation}
\mathcal A:f\mapsto -
\frac{\partial}{\partial x} \left[
\diag\left((a_i-b_i \cdot )f_i\right)_{i=1,\ldots,n}
 \right] + Mf,
\end{equation}
on $L^1(\frac{a}{b},\frac{c}{d};\R^n)$, where $M$ is an $L^\infty$-function taking values in the spaces of symmetric $n\times n$-matrices whose rows sum up to 0 and with off-diagonal entries that are a.e.\ bounded below away from 0.
It is easy to see that our generation result extend to this more general setting. We leave the details to the interested reader. {Different but comparable generalizations have been thoroughly discussed in~\cite{BenLebMal15}.}
\end{rem}

{We have seen in Section~\ref{sec:balance-1d} that, under the assumptions formulated in~\eqref{ass1} (cf.\ also the generalization discussed in Remark~\ref{rem:hill}), there is a non-zero stationary state, i.e., our operator $\mathcal A$ has non-trivial null space. More generally, b}y~\cite[Cor.~C.III.4.3]{Nag86} the eigenvalues of $\mathcal A$ that lie on $i\mathbb R$ form an additive cyclic group, i.e., 
\begin{equation}\label{eq:kerA}
\sigma_p(\mathcal A)=i\mathbb R\quad\hbox{or}\quad
\sigma_p(\mathcal A)=ik\mathbb Z\quad \hbox{for some }k\in \mathbb N_0\ .
\end{equation}

Thus, the semigroup  $(e^{t\mathcal A})_{t\ge 0}$ itself may a priori still converge towards a rotation group (i.e., a linear combination of terms $(e^{t\lambda})_{t\ge 0}$, $\lambda \in \sigma_p(\mathcal A)\cap i\mathbb R$) as $t\to 0$, as the subset $\sigma_p(\mathcal A)\cap i\mathbb R$ of the point spectrum of $\mathcal A$ that lies on the imaginary axis may contain non-zero elements. An additional compactness argument is needed to rule out this case, but this turns out to be more delicate than expected. 

First of all, observe that elements in the domain of $\mathcal A$ are $W^{1,1}$, and in particular $L^\infty$, on each compact sub-interval of $(\frac{a}{b},\frac{c}{d})$, hence in particular the resolvent operator $R(\lambda,\mathcal A)$ of $\mathcal A$ maps $L^1(\frac{a}{b},\frac{c}{d};\mathbb R^2)$ to $L^\infty_{loc}(\frac{a}{b},\frac{c}{d};\mathbb R^2)$ for all $\lambda >0$: by~\cite[Cor.~2.4]{Are08}, this implies that $R(\lambda,\mathcal A)$ is a kernel operator for all $\lambda >0$; we already know that it is positive as well. Now, let us recall a notion from the theory of operators on Banach lattices, based on the concept of \emph{AM-spaces}, cf.~\cite[\S~II.7]{Sch74}: a positivity preserving operator on $L^1(\frac{a}{b},\frac{c}{d};\mathbb R^2)$ is said to be \emph{AM-compact} if it maps order intervals
\[
[\eta,\theta]:=\left\{g\in L^1\left(\frac{a}{b},\frac{c}{d};\mathbb R^2\right):\eta\le g(x)\le \theta\hbox{for a.e. } x\in \left(\frac{a}{b},\frac{c}{d}\right)\right\}
\]
into precompact sets of $L^1\left(\frac{a}{b},\frac{c}{d};{\mathbb R^2}\right)$ for all $0\le \theta\in L^1(\frac{a}{b},\frac{c}{d};\mathbb R^2)$. It is known that all positive kernel operators are AM-compact, cf.~\cite[Prop.~A.1]{GerGlu17}.

\begin{lemma}\label{lemma:0eigenv}
The only {possible} eigenvalue of $\mathcal A = \mathcal{A}_1 + \mathcal{A}_2$ along the imaginary axis is $0$.
\end{lemma}

\begin{proof}
{Let the point spectrum of $\mathcal A$ intersect $i\mathbb R$: by~\eqref{eq:kerA}, the null space $\ker \mathcal A$ of $\mathcal A$ is non-trivial.}
By the Theorem of Perron--Frobenius for irreducible semigroups, cf.~\cite[Prop.~VI.3.4]{EngNag06}, there is a unique element of $\ker \mathcal A$ which is strongly positive and has norm $1$; let us denote it by 
$(\psi_1,\psi_2)$. Assume for a contradiction that $(f_1,f_2)$ is a (normalised) eigenvector of $\mathcal{A}$ for an eigenvalue $i\beta$ where $\beta \in \R \setminus \{0\}$. Since $\lambda$ and $\mu$ are continuous, it follows from $(f_1,f_2) \in \ker(i\beta - \mathcal{A})$ that $f_1$ and $f_2$ are $C^1_{loc}$-functions.
		We use a standard argument from Perron--Frobenius theory to show that $|(f_1,f_2)| = (|f_1|,|f_2|) \in \ker \mathcal{A}$: indeed, we have
	\begin{align*}
		|(f_1,f_2)| = |e^{t\mathcal{A}}(f_1,f_2)| \le e^{t\mathcal{A}}|(f_1,f_2)|
	\end{align*}
	for every $t \ge 0$. Since each semigroup operator $e^{t\mathcal{A}}$ is contractive and since the norm on $L^1(\frac{a}{b},\frac{c}{d};\C^2)$ is strictly monotone, it follows that actually $|(f_1,f_2)| = e^{t\mathcal{A}}|(f_1,f_2)|$ for all $t \ge 0$. Hence, $|(f_1,f_2)| \in \ker \mathcal{A}$. From this we conclude that $(|f_1|,|f_2|) = (\psi_1,\psi_2)$. We can therefore find functions $\gamma_1,\gamma_2: (\frac{a}{b},\frac{c}{d}) \to \R$ such that $f_k = \psi_k e^{i\gamma_k}$ for $k=1,2$. Since $e^{i\gamma_k}$ is in $C^1_{loc}$ (as $\psi_1$ and $\psi_2$ are $> 0$ on $(\frac{a}{b},\frac{c}{d})$), we can choose $\gamma_k$ to be in $C^1_{loc}$, too.
	
	Now, it follows from \cite[Thm.~C-III-2.2]{Nag86} that $(\psi_1e^{in\gamma_1}, \psi_2 e^{in\gamma_2}) \in \ker(in\beta - \mathcal{A})$ for all integers $n \in \Z$. Using the definition of $\mathcal{A}$, this yields after a short computation that
	\begin{align}
		\label{eq:eigenvalue-equality}
		\begin{split}
			\beta
			\begin{pmatrix}
				\psi_1 \\ \psi_2
			\end{pmatrix} =
			& -
			\begin{pmatrix}
				\gamma_1'(a-bx)\psi_1 \\ \gamma_2'(c-dx) \psi_2
			\end{pmatrix}
			+ \frac{i}{n}
			\begin{pmatrix}
				\frac{d}{dx}[(a-bx)\psi_2] \\ \frac{d}{dx}[(c-dx) \psi_2]
			\end{pmatrix}
			\\
			& - \frac{i}{n} 
			\begin{pmatrix}
				e^{-in\gamma_1} & 0 \\
				0 & e^{-in\gamma_2}
			\end{pmatrix}
			\begin{pmatrix}
				-\lambda & \mu \\
				\lambda & -\mu
			\end{pmatrix}
			\begin{pmatrix}
				\psi_1 e^{in\gamma_1} \\
				\psi_2 e^{in\gamma_2}
			\end{pmatrix}
		\end{split}
	\end{align}
	for all $n \in \Z$. By letting $n \to \infty$ we thus obtain that
	\begin{align}
		\label{eq:equality-independent-of-n}
		\beta
		\begin{pmatrix}
			\psi_1 \\ \psi_2 
		\end{pmatrix} =
		\begin{pmatrix}
			\gamma_1' (a-bx) \psi_1 \\ \gamma_2' (c-dx)\psi_2
		\end{pmatrix}.
	\end{align}
	Since $\psi_1$ and $\psi_2$ are $> 0$ on $(\frac{a}{b},\frac{c}{d})$, it follows that $\gamma_1' = \frac{\beta}{a-bx}$ and $\gamma_2' = \frac{\beta}{c-dx}$.
	
	On the other hand, by plugging~\eqref{eq:equality-independent-of-n} into~\eqref{eq:eigenvalue-equality} and using that $\mathcal{A}(\psi_1,\psi_2) = 0$, we obtain
	\begin{align*}
		\begin{pmatrix}
			-\lambda & \mu \\
			\lambda & -\mu 
		\end{pmatrix}
		\begin{pmatrix}
			\psi_1 \\ \psi_2
		\end{pmatrix} 
		= 
		\begin{pmatrix}
			e^{-in\gamma_1} & 0 \\
			0 & e^{-in\gamma_2}
		\end{pmatrix}
		\begin{pmatrix}
			-\lambda & \mu \\
			\lambda & -\mu 
		\end{pmatrix}
		\begin{pmatrix}
			\psi_1 e^{in\gamma_1} \\ \psi_2 e^{in\gamma_2}
		\end{pmatrix}.
	\end{align*}
	A short computation thus yields
	\begin{align*}
		\begin{pmatrix}
			\psi_2 \\ \psi_1
		\end{pmatrix}
		=
		\begin{pmatrix}
			\psi_2 e^{in(\gamma_2 - \gamma_1)} \\
			\psi_1 e^{in(\gamma_1 - \gamma_2)}
		\end{pmatrix}
	\end{align*}
	for all $n \in \Z$. Plugging in $n = 1$ and using again that $\psi_1$ and $\psi_2$ are $> 0$ on $(\frac{a}{b},\frac{c}{d})$, we thus conclude that $\gamma_1-\gamma_2$ is constant. Hence, $\gamma_1' = \gamma_2'$ and by using the expressions we derived for $\gamma_1'$ and $\gamma_2'$ above, we thus obtain $\frac{\beta}{a-bx} = \frac{\beta}{c-dx}$.
	
	Now we use the assumption that $\beta \not= 0$: this implies that $a-bx = c-dx$ (for all $x \in (\frac{a}{b},\frac{c}{d})$), and therefore $a = c$ and $b = d$. Hence $\frac{a}{b} = \frac{c}{d}$, which is a contradiction.
\end{proof}

Now, we are finally in the position of stating our main result in this section.
\begin{theo}\label{thm:joch}
The semigroup $(e^{t\mathcal A})_{t\ge 0}$ converges strongly towards the projector onto $\ker \mathcal A$.
\end{theo}

Our main step in the proof is to check that the semigroup is relatively compact in the strong operator topology. Relative compactness of the orbits merely in the \textit{weak} operator topology
would by~\cite[Cor.~V.4.6]{EngNag00} already imply mean ergodicity of the semigroup: but our result in Theorem~\ref{thm:joch} is significantly stronger, since convergence to equilibrium is actually achieved for individual orbits not only in a time-averaged sense.

\begin{proof}
In view of Lemma~\ref{lemma:0eigenv}, the claim is an immediate consequence of~\cite[Thm.~V.2.14]{EngNag00} if we can prove that the orbits of the semigroup are relatively compact.

To this purpose, let $f_0$ be a  strongly positive vector in the null space of $\mathcal A$ (and hence also a fixed point of $e^{t\mathcal A}$ for all $t\ge 0$), so that in particular $f_0=f_0-{\mathcal A}f_0$, i.e., $R(1,\mathcal A)f_0=f_0$. Introduce the set
\[
L^1_{f_0}:=\left\{L^1\left(\frac{a}{b},\frac{c}{d};{\mathbb R^2}\right):\exists c\in \mathbb R_+:|f|\le c|f_0| \right\}\ ,
\]
which is dense in $L^1\left(\frac{a}{b},\frac{c}{d};{\mathbb R^2}\right)$, thus implying that $R(1,\mathcal A)\left(L^1_{f_0}\right)$ is dense in $R(1,\mathcal A)\left(L^1\left(\frac{a}{b},\frac{c}{d};{\mathbb R^2}\right)\right)=D(\mathcal A)$ and hence in $L^1\left(\frac{a}{b},\frac{c}{d};{\mathbb R^2}\right)$.

Hence, in order to prove relative compactness of the orbits we can take without loss of generality some $f=R(1,\mathcal A)g$ that additionally satisfies $|g|\le cf_0$. Then by using positivity of $e^{t\mathcal A}$ we see that
\[
|e^{t\mathcal A}f|\le e^{t\mathcal A}|f|\le e^{t\mathcal A}R(1,\mathcal A)|g|\le R(1,\mathcal A)ce^{t\mathcal A}f_0 =R(1,\mathcal A)cf_0
\]
for all $t\ge 0$, and therefore
\[
e^{t\mathcal A}f\in \left[-cR(1,\mathcal A)f_0,cR(1,\mathcal A)f_0\right]\subset R(1,\mathcal A)[-cf_0,cf_0]\ \quad \hbox{for all }t\ge 0\ .
\]
All in all, we have proved that the orbits of the semigroup are contained in the image under $R(1,\mathcal A)$ of some order interval: because of AM-compactness of $R(1,\mathcal A)$, these sets are hence relatively compact.
\end{proof}

{
\begin{rem}\label{rem:hill2}
We have deliberately formulated Lemma~\ref{lemma:0eigenv} and Theorem~\ref{thm:joch} in a rather general way: their proof does not depend on the fact that $0$ is an eigenvalue. In view of Remark~\ref{rem:hill}, this can be rephrased by saying that if the conditions~\ref{ass1} are relaxed, e.g.\ be letting $\lambda,\mu$ be Hill-type functions, our system is still well-posed and in fact governed by a semigroup on $L^1(\frac{a}{b},\frac{c}{d};\R^2)$ that converges strongly to 0.
\end{rem}
}

\begin{rem}\label{rem:irred}
A close inspection of the proofs of Proposition~\ref{lem:semigracal} and Lemma~\ref{lemma:0eigenv} shows that Theorem~\ref{thm:joch} holds not only for $\mathcal A_2$ as in~\eqref{eq:a1a2}, but more generally under the assumptions that $\mathcal A_2$  is a bounded linear operator on $L^1(\frac{a}{b},\frac{c}{d};\R^2)$ such that
\begin{itemize}
\item $\mathcal A_2$ generates a positive, irreducible semigroup with ${\mathcal A_2}^T{\mathbb 1}=0$; and
\item ${\mathcal D}{\mathcal A}_2-{\mathcal A}_2{\mathcal D}= \Phi(e^{in\gamma_1},e^{in\gamma_2})\mathcal J$
for some injective bounded linear operator $\mathcal J$ and some function $\Phi:\C^2\to\C^2$ that only vanishes along the diagonal $\{(x,x):x\in \C\}$, where
\[
{\mathcal D}:=\begin{pmatrix}
e^{in\gamma_1} & 0\\ 0 & e^{in\gamma_2}
\end{pmatrix}\ .
\]
(The first conditions is equivalent to requiring the semigroup generated by $\mathcal A_2$ to be stochastic and irreducible.)	
\end{itemize}
\end{rem}

{
Under our standing assumptions~\eqref{ass1} on the switching rates $\lambda, \mu$,} we sum up our findings as follows.

\begin{cor}\label{cor:convergenceafter}
The null space of $\mathcal A$ is spanned by a  strongly positive function $\psi$ and the semigroup generated by $\mathcal A$ on $L^1(\frac{a}{b},\frac{c}{d};\mathbb R^2)$ converges strongly towards the orthogonal projector
\[
f\mapsto \int_{\frac{a}{b}}^\frac{c}{d} \left(f(x)|\phi(x)\right)_{\mathbb R^2}\diff x \cdot \psi\ ,
\]
where $\varphi$ is the  strongly  positive function that spans $\Ker \mathcal A'$ and such that $\int_{\frac{a}{b}}^\frac{c}{d} \left(\psi(x)|\phi(x)\right)_{\mathbb R^2}\diff x=1$.
\end{cor}

\begin{rem}
One may expect that the long-time behavior of~\eqref{eq:system} can be investigated by applying classical results in the theory of positivity preserving operators like~\cite[Exer.~II.4.21.(2)]{EngNag00}, \cite[Thm.~VI.3.5]{EngNag06} or \cite[Thm.~12]{Dav05}. However, we have not been able to prove that the semigroup that governs our problem is
\begin{itemize}
\item \emph{eventually norm continuous} 
(cf.~\cite[Def.~II.4.17]{EngNag00}),
\item or  \emph{quasi-compact} (cf.~\cite[Def.~V.4.4]{EngNag06}),
\item or it has the \emph{Feller property} (i.e., the solutions to~\eqref{eq:system} are continuous for all $t>0$ and each initial data in $L^1$);
\end{itemize}
in fact, we doubt that these properties hold at all. We also observe that if the semigroup could be proved to satisfy $\inf\{e^{t\mathcal A},e^{s\mathcal A}\}>0$ for some $t>s\ge 0$, then by~\cite[Cor.~C-IV.2.10]{Nag86} the convergence stated in Corollary~\ref{cor:convergenceafter} would hold in operator norm, too.
\end{rem}

Needless to say, the function $\psi$ in Corollary~\ref{cor:convergenceafter} is nothing but the  function $\psi$ obtained in Theorem~\ref{thm:pavel-densities}: this stationary solution is uniquely determined by the initial data $f_0$ and the parameters $\lambda,\mu$.
We will present more explicit formulae for the  strongly positive function $\psi$ in the following section, for special choices of $\lambda,\mu$.

\section{Analytical solutions and numerical examples} \label{SecLast}

Let us discuss the simplest case: $\lambda $ and $ \mu $ are assumed to be affine functions, i.e.,
\begin{equation}
\begin{array}{ccc}
\displaystyle  \lambda (x) & = & \displaystyle  lx+k\ , \\
 \displaystyle  \mu (x) & = & \displaystyle m x +n\ ,
 \end{array}
 \end{equation}
for some $k,l,m,n\in \mathbb R$ chosen arbitrarily subject to guarantee positivity of $ \lambda$ and $ \mu$: then all integrals in~\eqref{solution1} can be explicitly calculated.

We then have
$$
 \frac{l x + k}{bx-a}  + \frac{mx + n}{dx-c}  = \frac{l}{b} + \frac{l}{b} \frac{\frac{a}{b} + k/l}{x-\frac{a}{b}} 
 + \frac{m}{d} + \frac{m}{d} \frac{\frac{c}{d} + n/m}{x-\frac{c}{d}} $$
and we get
\begin{equation}
h(x) = K e^{\left(\frac{l}{b} + \frac{m}{d} \right) x} (x-\frac{a}{b})^{\frac{al}{b^2} + \frac{k}{b}} (\frac{c}{d}-x)^{\frac{cm}{d^2} + \frac{n}{d}}.
\end{equation}
The solution to the system~\eqref{eq:stationary} is given by
\begin{equation}
\begin{array}{ccl}
\psi_1 (x) & = & \displaystyle  \frac{K}{b} e^{\left(\frac{l}{b} + \frac{m}{d} \right) x}  \Big( x-\frac{a}{b} \Big)^{\frac{al}{b^2} + \frac{k}{b}-1} \Big(\frac{c}{d}-x \Big)^{\frac{cm}{d^2} + \frac{n}{d}}, \\[5mm]
\psi_2 (x) & = &  \displaystyle  \frac{K}{d} e^{\left(\frac{l}{b} + \frac{m}{d} \right) x} \Big(x-\frac{a}{b} \Big)^{\frac{al}{b^2} + \frac{k}{b}} \Big(\frac{c}{d}-x \Big)^{\frac{cm}{d^2} + \frac{n}{d}-1}. \end{array}
\end{equation}

We see that the functions satisfy the boundary conditions~\eqref{eq:balance}, since
\begin{equation} \frac{cm}{d^2} + \frac{n}{d} = \frac{1}{d} \mu \Big(\frac{c}{d}\Big)   > 0 \; \; \; \mbox{and} \; \; \; \frac{al}{b^2} + \frac{k}{b} = \frac{1}{b} \lambda \Big(\frac{a}{b} \Big)  > 0 . 
\end{equation}

We have thus obtained an explicit, analytic solution to the stationary differential equation~\ref{eq:stationary}. This allows us to analyze the behavior
of the solutions depending on the values of the parameters:
\begin{itemize}
\item if $ al + kb < b^2 $, then $ \psi_1 $ is singular at $ x = \frac{a}{b},$
\item if $ al + kb = b^2 $, then $ \psi_1 $ attains a nonzero value at $ x = \frac{a}{b}, $
\item if $ al + kb > b^2 $, then $ \psi_1 $ tends to zero at $ x = \frac{a}{b}; $ \\
\item if $ cm + nd < d^2 $, then $ \psi_2 $ is singular at $ x = \frac{c}{d}, $
\item if $ cm + nd = d^2 $, then $ \psi_2 $ attains a nonzero value at $ x = \frac{c}{d},$
\item if $ cm + nd > d^2 $, then $ \psi_2 $  tends to zero at $ x = \frac{c}{d}. $
\end{itemize}

Let us check where the maxima of $\psi_1,\psi_2$ are situated in the case 
where these functions are not singular at one of the endpoints. We calculate first the
derivative of $ \psi_1 $
\begin{equation}\label{f11}
\begin{split}
\psi_1'(x)&=\frac{K}{b}e^{\left(\frac{l}{b}+\frac{m}{d}\right)x}
\left(x-\frac{a}{b} \right)^{\left(\frac{al}{b^2}+\frac{k}{b}-2\right)}
\left(\frac{c}{d} -x \right)^{\left(\frac{cm}{d^2}+\frac{n}{d}-1\right)}\cdot\\
&\cdot 
\underbrace{\left[\left(\frac{l}{b}+\frac{m}{d}\right)\left(x-\frac{a}{b}\right)\left(\frac{c}{d}-x\right)-\left(x-\frac{a}{b}\right)\left(\frac{cm}{d^2}+\frac{n}{d}\right)+\left(\frac{c}{d}-x\right)\left(\frac{al}{b^2}+\frac{k}{b}-1\right) \right]}_{
\displaystyle := P_1(x)}\\
&=:\frac{K}{b}e^{\left(\frac{l}{b}+\frac{m}{d}\right)x}\left(x-\frac{a}{b} \right)^{\left(\frac{al}{b^2}+\frac{k}{b}-2\right)}
\left(\frac{c}{d} -x \right)^{\left(\frac{cm}{d^2}+\frac{n}{d}-1\right)}\cdot P_1(x)\ ,
\end{split}
\end{equation}
where $ P_1(x) $ is a quadratic polynomial. 
In the general position case
$$ al + kb > b^2 \; \; \mbox{and} \; \; cm + nd \geq d^2\   $$
the function $ \psi_1 $  is a non-negative function 
vanishing in the boundary points $\frac{a}{b},\frac{c}{d}$, hence it must have an odd number of local maxima 
in the open interval $(\frac{a}{b},\frac{c}{d})$. But its critical points
 come from the zeroes of the quadratic polynomials $ P_1 (x),$
 having at most two zeroes. Hence the function $ \psi_1 $ has exactly one local maximum inside the interval.
 This maximum is necessarily global.
 
 In the border case 
 $$ al + kb = b^2 \; \; \mbox{and} \; \; cm + nd \geq d^2\   $$
 the polynomial takes the form
 $$ 
\begin{array}{ccl}
P_1(x) & = &  \left(\frac{l}{b}+\frac{m}{d}\right)\left(x-\frac{a}{b}\right)\left(\frac{c}{d}-x\right)-\left(x-\frac{a}{b}\right)\left(\frac{cm}{d^2}+\frac{n}{d}\right)+\left(\frac{c}{d}-x\right) \underbrace{\left(\frac{al}{b^2}+\frac{k}{b}-1\right)}_{= 0}\\
& = &  \left[ \left(\frac{l}{b}+\frac{m}{d}\right) \left(\frac{c}{d}-x\right)-\left(\frac{cm}{d^2}+\frac{n}{d}\right) \right] \left(x-\frac{a}{b}\right).
\end{array}$$
One of the roots coincides with the endpoint $ x = \frac{a}{b} ,$ hence $ \psi_1 $ has at most one critical point
inside the interval. This point cannot be minimum point, hence the function $ \psi_1 $ has
exactly one maximum inside the semi-closed interval $ [\frac{a}{b}, \frac{c}{d} ). $
The maximum is global and may coincide with the left endpoint $ \frac{a}{b}. $ This occurs
if
$$ \left(\frac{l}{b}+\frac{m}{d}\right) \left(\frac{c}{d}-\frac{a}{b} \right)-\left(\frac{cm}{d^2}+\frac{n}{d}\right)  = \frac{lc}{bd}- \frac{la}{b^2} - \frac{ma}{bd}-\frac{n}{d} \leq 0  $$

Similar analysis can be applied for the function $ \psi_2 $
$$ \psi_2' (x) =   \frac{K}{d} e^{\left(\frac{l}{b} + \frac{m}{d} \right) x} \Big(x-\frac{a}{b} \Big)^{\frac{al}{b^2} + \frac{k}{b}-1 } \Big(\frac{c}{d}-x \Big)^{\frac{cm}{d^2} + \frac{n}{d}-2} P_2 (x), $$
where $ P_2 $ is quadratic polynomial as well. In the general position case
$$ al + kb \geq b^2 \; \; \mbox{and} \; \; cm + nd > d^2\   $$
there is precisely one maximum inside the open interval. In the border case
$$ al + kb \geq  b^2 \; \; \mbox{and} \; \; cm + nd = d^2 , $$
there is one maximum in the semi-closed interval $ (\frac{a}{b}, \frac{c}{d} ]
, $ Note that the maximum may coincide with the right endpoint.

Let us turn to concrete examples  of one-fluid systems. We compared 
   our analytic solution to numerical 
simulations for several examples (some use parameters which are biologically not meaningful). 
For each example we provide the values
of the parameters, analytic formula for the solutions and plot the results.
Since the numerical and analytical solutions are indistinguishable, we show only a single plot 
for each example. 
Moreover, we provide the values of the parameters, formulas for the normalized solutions and plots of these
functions. Each time the functions $\psi_1$ and $\psi_2$ are plotted in green and blue, respectively, while their sum -- the target probability distribution of the system -- is plotted in red.
All computations were done using Matlab\textsuperscript{\tiny\textregistered}.

\noindent \underline{\bf Example 1} 
\newline
Parameter values:
$$ a = -1, \; b = d = c = 1, \; \lambda \equiv 1,\; \mu \equiv 1 $$
Normalized solution:
$$ \left\{
\begin{array}{ccc}
\psi_1 (x) &  = & (1-x)/4, \\[2mm]
\psi_2(x) & = & (x+1)/4, \\[2mm]
\psi_1(x) + \psi_2(x) & = & 1/2.
\end{array} \right. $$

\begin{figure}[htb]
\includegraphics[height=2in]{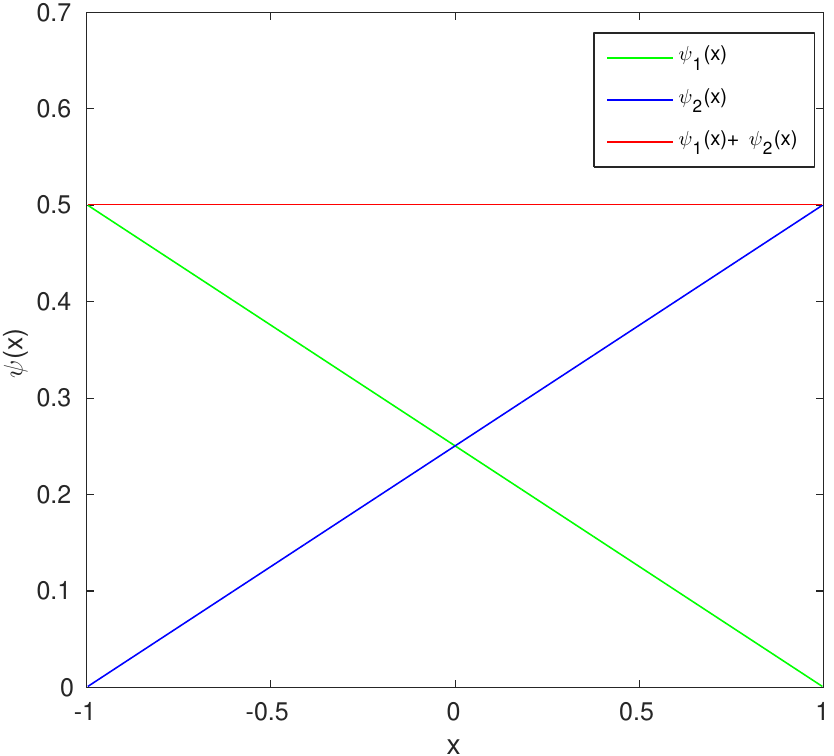}
\caption{Solution of Example 1.}

\label{Ffig1}
\end{figure}

\noindent \underline{\bf Example 2} 
\newline
Parameter values:
$$ a = -1, \;  b = d = c = 1, \; \lambda \equiv 2,\; \mu \equiv 2 $$
Normalized solution:
$$ \left\{
\begin{array}{ccl}
\psi_1 (x) &  = & 3 (x+1)^2 (1-x)/8 \\[2mm]
\psi_2(x) & = & 3 (x+1)(1-x)^2/8 \\[2mm]
\psi_1(x) + \psi_2(x) & = &  3 (1-x^2)/4
\end{array} \right. $$

\begin{figure}[htb]
\includegraphics[height=2in]{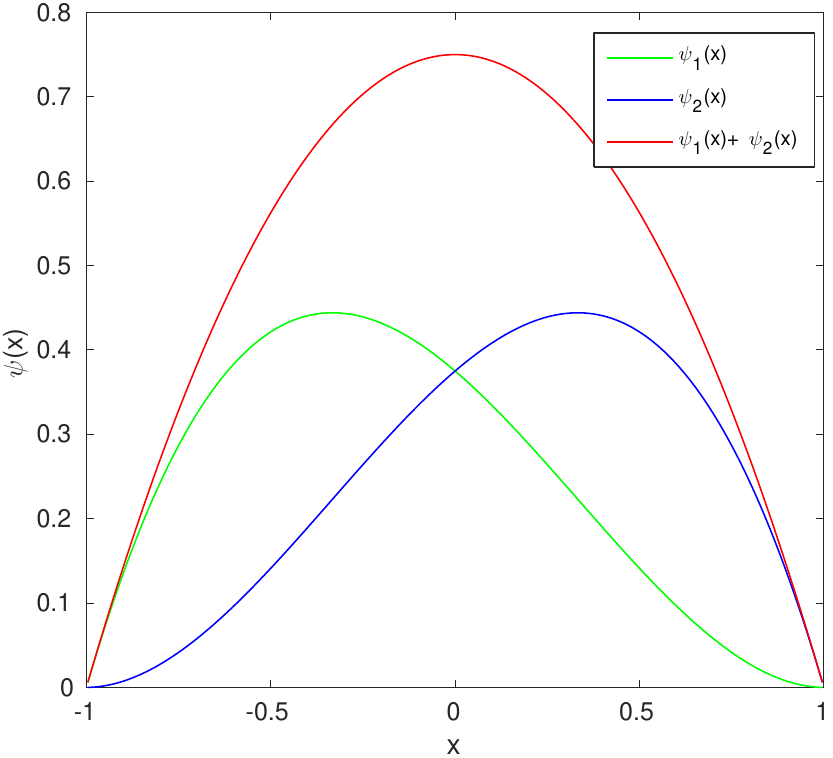}
\caption{Solution of Example 2.}

\label{Ffig2}
\end{figure}

\noindent \underline{\bf Example 3} 
\newline
Parameter values:
$$ a = -1, \;  b = d = c = 1, \; \lambda \equiv 1/2,\; \mu \equiv 1/2 $$
Normalized solution 
$$ \left\{
\begin{array}{ccl}
\psi_1 (x) &  = &  \displaystyle \frac{1}{2 \pi}  \frac{\sqrt{1-x}}{\sqrt{x+1}} \\[2mm]
\psi_2(x) & = &  \displaystyle \frac{1}{2 \pi}  \frac{\sqrt{x+1}}{\sqrt{1-x}} \\[2mm]
\psi_1(x) + \psi_2(x) & = & \displaystyle \frac{1}{\pi}  \frac{1}{\sqrt{1-x^2}} 
\end{array} \right. $$

\begin{figure}[htb]
\includegraphics[height=2in]{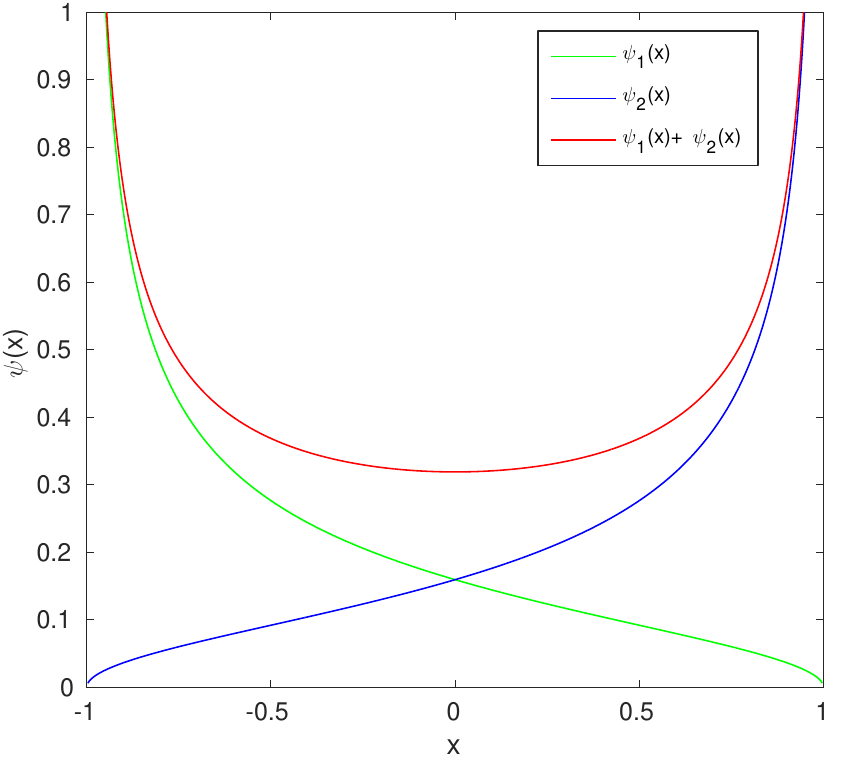}
\caption{Solution of Example 3.}

\label{Ffig3}
\end{figure}

\bigskip
Let us now turn to the case of affine functions $\lambda(x)=lx+k$, $\mu(x)=mx+n$.

\noindent \underline{\bf Example 4} 
\newline
Parameter values:
$$ a = -1, \;  b = d = c = 1, \; l = 0, k= 1, \; m = -1, n=3$$

Normalized solution
$$ \left\{
\begin{array}{ccl}
\psi_1 (x) &  = &  \displaystyle \frac{e}{2(e^2+1)} e^{-x} (1-x)^2 \\[2mm]
\psi_2(x) & = &  \displaystyle \frac{e}{2(e^2+1)}  e^{-x} (1-x^2)  \\[2mm]
\psi_1(x) + \psi_2(x) & = & \displaystyle \frac{e}{e^2+1} e^{-x} (1-x)
\end{array} \right. $$


\begin{figure}[htb]
\includegraphics[height=2in]{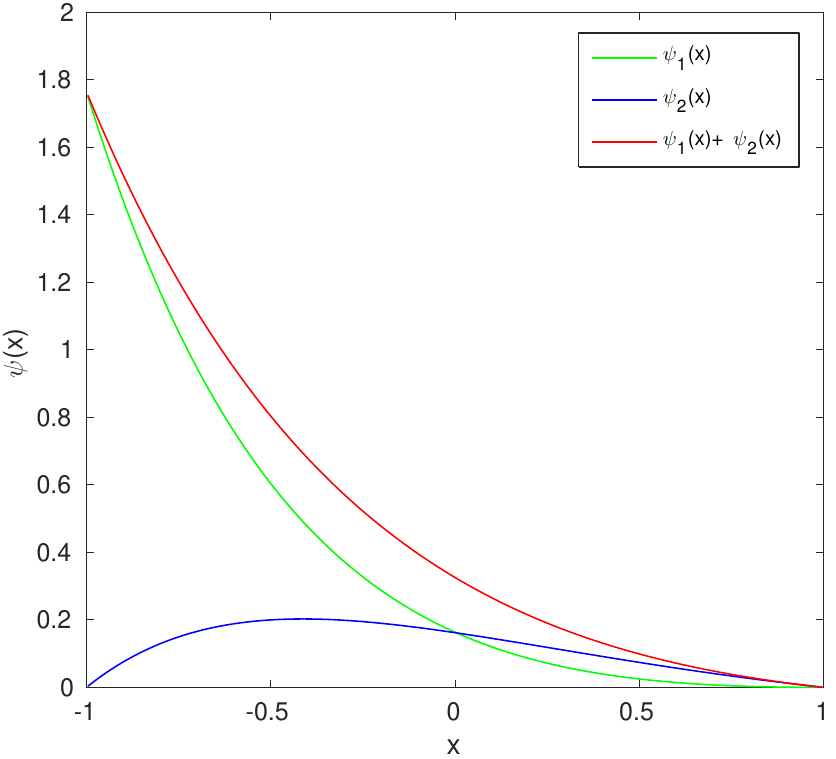}
\caption{Solution of Example 4.}

\label{Ffig4}
\end{figure}

\noindent \underline{\bf Example 5} 
\newline
Parameter values:
$$ a = 0, \;  b = 1,\; c = 2,\;  d =  1, \; l = -2, k= 4, \; m = 1, n=0$$
Normalized solution
$$ \left\{
\begin{array}{ccl}
\psi_1 (x) &  = &  \displaystyle \frac{e^2}{8(23 -3 e^2)} e^{-x} x^{3} (2-x)^{2} \\[2mm]
\psi_2(x) & = &  \displaystyle  \frac{e^2}{8(23 -3 e^2)}  e^{-x} x^4 (2-x) \\[2mm]
\psi_1(x) + \psi_2(x) & = & \displaystyle \frac{e^2}{4(23 -3 e^2)} e^{-x} x^3 (2-x)
\end{array} \right. $$


\begin{figure}[!htb]
\includegraphics[height=2in]{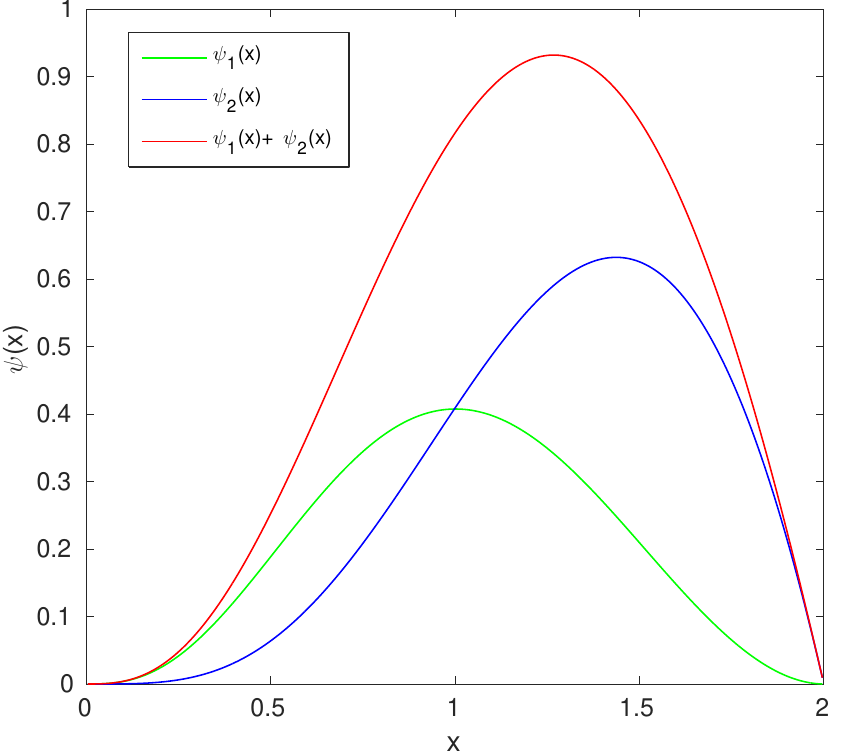}
\caption{Solution of Example 5.}

\label{Ffig5}
\end{figure}

\noindent \underline{\bf Example 6} 
\newline
Parameter values:
$$ a = -1, \;  b = c =  d =  1, \; l = 0, k= 1, \; m = 2, n=1$$
Normalized solution
$$ \left\{
\begin{array}{ccl}
\psi_1 (x) &  = &  \displaystyle (-6 \cosh 2 + 7 \sinh 2)^{-1} e^{2 x}  (1 - x)^3 \\[2mm]
\psi_2(x) & = &  \displaystyle (-6 \cosh 2 + 7 \sinh 2)^{-1} e^{2 x}(x + 1) (1 - x)^2 \\[2mm]
\psi_1(x) + \psi_2(x) & = & \displaystyle 
(-6 \cosh 2 + 7 \sinh 2)^{-1} ({2e^{2 x}  (1 - x)^2})
\end{array} \right. $$

\begin{figure}[!htb]
\includegraphics[height=2in]{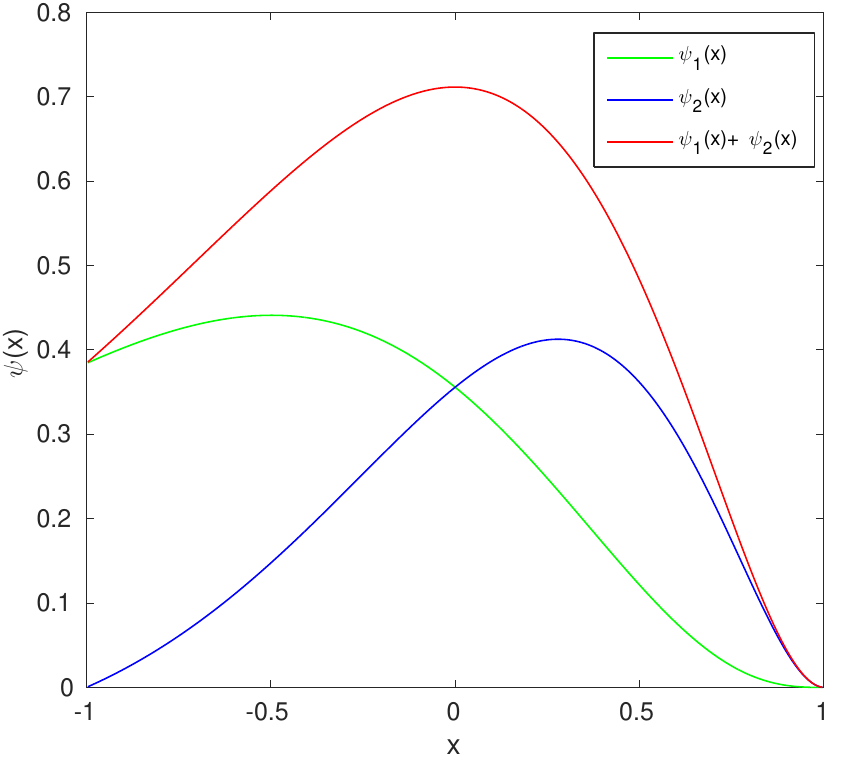}
\caption{Solution of Example 6.}

\label{Ffig6}
\end{figure}

\noindent \underline{\bf Example 7} 
\newline
Parameter values:
$$ a = -1, \;  b = c =  d =  1, \; l = 0, k= 1, \; m = 1, n=2$$
Normalized solution
$$ \left\{
\begin{array}{ccl}
\psi_1 (x) &  = &  \displaystyle \frac{8 e^2}{19 - 304 e + 48 e^3 + e^4}
 e^{x}  (1 - x)^3 \\[2mm]
\psi_2(x) & = &  \displaystyle  \frac{8 e^2}{19 - 304 e + 48 e^3 + e^4}
 e^{x}(x + 1) (1 - x)^2 \\[2mm]
\psi_1(x) + \psi_2(x) & = & \displaystyle 
 \frac{8 e^2}{19 - 304 e + 48 e^3 + e^4} ({2e^{x}  (1 - x)^2})
\end{array} \right. $$

\begin{figure}[htb]
\includegraphics[height=2in]{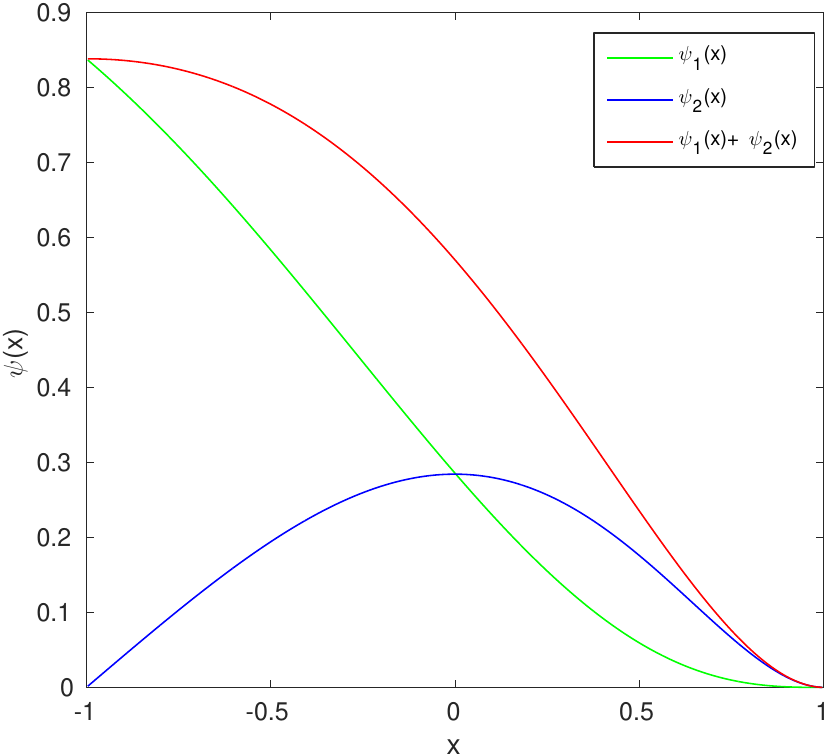}
\caption{Solution of Example 7.}

\label{Ffig7}
\end{figure}

In all considered examples the functions $ \lambda $ and $ \mu $ were assumed to be
affine for simplicity. Even this choice provided us with a rather rich class of models. Note that
any other simple analytic expression will do the job. One should only assure that the
integrals in~\eqref{solution1} can be calculated analytically. The last two examples were chosen to 
illustrate the power of analytic calculations: in both cases $ \psi_1 $ does not approach zero and
is not singular at the left endpoint. In one case it has a maximum inside the interval, in the other case
it is monotonically decreasing. It was easy to find parameters leading to such behavior since analytic
formulas were available, this would be a challenging task if only computer simulations were
available.

\section{Conclusion}

We considered a hybrid model of a self-regulating gene, which is a common motif in gene regulatory networks. Our model describes the evolution of the discrete random state (mode) of the gene (``on" or ``off") and the corresponding continuous protein concentration. The latter evolves according to an ordinary differential equation and leads to a system of PDEs for the evolution of two probability densities (one for each mode). Assuming that the rate functions $\lambda$ and $\mu$ for mode changing are known explicitly, we analyzed the properties of the PDE system and studied well-posedness in an $L^1$-setting. Exploiting the theory of positive operator semigroups we  rigorously proved convergence towards stationary solutions in strong operator topology and derived an analytic expression for such stationary densities. 
Our solution is valid for a large class of protein production and degradation rates and structurally much simpler and easier to evaluate than the solution of the corresponding fully discrete master equation model given by Grima et al. \cite{GriSchNew12}. 
As future work, we plan to investigate the extension of this gene feedback loop two or more interacting gene and their corresponding proteins such as the exclusive or toggle switch \cite{lipshtat2006genetic,loinger2007stochastic}. In these cases, the support of the stationary solution of the corresponding hybrid model has a more complex shape. For instance, in the case of two genes and three modes (as one mode is not reachable), the density is only non-zero within a triangle whose endpoints are determined by the mode-conditional equilibria of the two protein concentrations. Although this is straightforward to see from Monte-Carlo simulations of the model, proving this and other properties for the corresponding PDE system is challenging.

\section{Acknowledgements}  This work is the outcome of a collaboration between two mathematicians working on evolutionary
equations (DM) and partial differential equations (PK) and a computer biologist (VW). It has partially been carried out at the ZiF (Center for Interdisciplinary Research) in Bielefeld  in the framework of the Cooperation Group {\it Discrete and continuous models in the theory of networks}. The authors are indebted to the ZiF for financial support and hospitality.
The work of PK was also partially supported by the Swedish Research Council (Grant D0497301).
 The work of DM was partially supported by the German Research Foundation (Grant 397230547). The work of VW was partially supported by the German Research Foundation (Grant 391984329). We warmly thank Jochen Glück (Ulm) for suggesting to us the proofs of Lemma~\ref{lemma:0eigenv} and Theorem~\ref{thm:joch} and for many helpful discussions. {We also thank the anonymous referees for pointing us to relevant references. In particular, we have learned that some of our analytic formulae for stationary distributions have already been obtained in~\cite{faggionato2009non,zeiser2010autocatalytic}.}

\end{document}